\documentclass[11pt,a4paper]{amsart}
\usepackage[english]{babel}
\usepackage{amssymb,xspace}
\usepackage{amstext}
\theoremstyle{plain}
\usepackage{amsbsy,amssymb,amsfonts,latexsym}

\usepackage{enumerate}
\usepackage{graphics}

\title{A new class of frequently hypercyclic operators}

\author{Sophie Grivaux}
\address{
Laboratoire Paul Painlev\' e, UMR 8524, Universit\'e  Lille 1, Cit\' e Scientifique, 59655 Villeneuve d'Ascq
Cedex, France}
\email{grivaux@math.univ-lille1.fr}

\thanks{This work was partially supported by ANR-Projet Blanc DYNOP}

\def\T{\ensuremath{\mathbb T}}

\def\E{\ensuremath{\mathbb E}}

\def\C{\ensuremath{\mathbb C}}
\def\Q{\ensuremath{\mathbb Q}}
\def\N{\ensuremath{\mathbb N}}
\def\P{\ensuremath{\mathbb P}}

\newcommand{\sep}{separable}

\newcommand{\hy}{hypercyclic}
\newcommand{\fhy}{frequently hypercyclic}
\newcommand{\ops}{operators}
\newcommand{\op}{operator}
\newcommand{\st}{Steinhaus}
\newcommand{\ind}{independent}

\newcommand{\erg}{ergodic}
\newcommand{\eve}{eigenvector}
\newcommand{\eva}{eigenvalue}
\newcommand{\ps}{perfectly spanning set of eigenvectors associated to unimodular
eigenvalues}
\newcommand{\wrt}{with respect to}
\newcommand{\bs}{backward shift}

\newcommand{\ga}{Gaussian}

\newcommand{\inv}{invariant}

\newcommand{\mea}{measure}
\newcommand{\nd}{non-degenerate}

\newcommand{\wmx}{weak-mixing}

\newcommand{\HC}{Hypercyclicity Criterion}
\newcommand{\ifff}{if and only if}

\newcommand{\pss}[2]{\ensuremath{{\langle #1,#2\rangle}}}

\newtheorem{theorem}{Theorem}[section]

\newtheorem{lemma}[theorem]{Lemma}

\newtheorem{proposition}[theorem]{Proposition}

{\theoremstyle{definition}}

{\theoremstyle{definition}}

{\theoremstyle{definition}\newtheorem{definition}[theorem]{Definition}}

{\theoremstyle{definition}}

{\theoremstyle{definition}\newtheorem{remark}[theorem]{Remark}}

\newtheorem{question}[theorem]{Question}

{\theoremstyle{definition}\newtheorem*{FFC Criterion}{Frequent
Faber-hypercyclicity Criterion}}
\newtheorem*{Hypercyclicity Criterion}{Hypercyclicity Criterion}
{\theoremstyle{definition}\newtheorem*{GS Criterion}{Godefroy-Shapiro
Criterion}}
\def\piednote#1{\let\oldfn=\thefootnote\def\thefootnote{}\footnote{\noindent#1}%
\addtocounter{footnote}{-1}\def\thefootnote{\oldfn}}

\begin{document}

\begin{abstract}
We study in this paper a hypercyclicity property of linear dynamical systems: a bounded linear operator $T$ acting on a separable infinite-dimensional Banach space $X$ is said to be \emph{hypercyclic} if there exists a vector $x\in X$ such that $\{T^{n}x \textrm{ ; } n\geq 0\}$ is dense in $X$, and \emph{frequently hypercyclic} if there exists $x\in X$ such that for any non empty open subset $U$ of $X$, the set $\{n\geq 0 \textrm{ ; } T^n x\in U\}$ has positive lower density. We prove in this paper that if $T\in\mathcal{B}(X)$ is an operator which has ``sufficiently many'' eigenvectors associated to eigenvalues of modulus $1$ in the sense that these eigenvectors are perfectly spanning, then $T$ is automatically frequently hypercyclic.
\end{abstract}
\maketitle

\section{Introduction}
Let $X$ be a complex infinite-dimensional \sep\ Banach space, and $T$ a bounded linear \op\ on $X$. We are concerned in this paper with the dynamics of the \op\ $T$, i.e. with the behaviour of the orbits $\mathcal{O}\textrm{rb}(x,T)=\{T^{n}x \textrm{ ; } n\geq 0\}$, $x\in X$, of the vectors of $X$ under the action of $T$. Our main interest here will be in strong forms of hypercyclicity: recall that a vector $x\in X$ is said to be \emph{\hy}\ for $T$ if its orbit under the action of $T$ is dense in $X$. In this case the \op\ $T$ itself is said to be \hy. This notion of \hy ity as well as related matters in linear dynamics have been intensively studied in the past years. We refer the reader to the recent book \cite{BM} for more information on these topics.
\par\smallskip
Our starting point for this work are the papers \cite{BG1}, \cite{BG2} and \cite{BG3}, which study the role of the unimodular point spectrum in linear dynamics. By unimodular point spectrum of the \op\ $T$, we mean the set of \eva s of $T$ which are of modulus $1$.
That the behaviour of the eigenvectors of an operator has an influence on its hypercyclicity properties was first discovered by Godefroy and Shapiro in \cite{GS}: their work deals with eigenvectors associated to eigenvalues of modulus strictly larger than $1$ and strictly smaller than $1$. The eigenvectors associated to eigenvalues of modulus $1$ first appeared in the works of Flytzanis \cite{Fl} and Bourdon and Shapiro \cite{Sh}. 
Then it was shown in \cite{BG1} that if $T$ has ``sufficiently many \eve s associated to unimodular \eva s'' (precise definitions will be given later on) then $T$ is \hy. In \cite{BG2} and \cite{BG3} this study is pushed further on in the direction of ergodic theory: under some assumptions bearing either on the geometry of the underlying space $X$ or on the regularity of the \eve\ fields of the \op\ $T$, it is proved that $T$ admits a \nd\ \inv\ \ga\ \mea\ \wrt\ which it is \erg\ (even \wmx). Then a straightforward application of Birkhoff's \erg\ theorem shows that $T$ is ``more than hypercyclic'': it is \emph{\fhy}, i.e. there exists a vector $x\in X$ such that for every non-empty open subset $U$ of $X$, the set $\{n\geq
  0 \textrm{ ; } T^{n}x\in U\}$ of instants when the iterates of $x$ under $T$ visit $U$ has positive lower density. Such a vector $x$ is called a \fhy\ vector for $T$. Frequent \hy ity is a much stronger notion than \hy ity, and some \ops\ are \hy\ without being \fhy: an example is the Bergman \bs\ \cite{BG2}, and then it was proved in \cite{S} that no \hy\ \op\ whose spectrum has an isolated point can be \fhy. Thus, although every infinite-dimensional separable Banach space supports a \hy\ \op\ (\cite{A},\cite{B}), there are spaces on which there are no \fhy\ \ops. Nonetheless, quite a large number of \hy\ \ops\ are \fhy, at least on Hilbert spaces (see for instance \cite{BG2}, \cite{BoGE}). One of the tools which are used to prove the frequent hypercyclicity of an operator is the ergodic-theoretic argument mentioned above: it shows that as soon as $T$ has sufficiently many \eve s associated to unimodular \eva s, $T$ is \fhy. 
\par\smallskip
More precisely, let us recall the following definition from \cite{BG1} and \cite{BG2}, which quantifies the fact that $T$ admits ``plenty'' \eve s associated to \eva s lying on the unit circle $\T=\{\lambda \in\C \textrm{ ; } |\lambda |=1\}$:

\begin{definition}\label{def1}
We say that a bounded operator $T$ on $X$ has a \ps\ if there exists a continuous probability measure $\sigma  $ on the unit circle $\T$ such that for every $\sigma  $-measurable subset $A$ of $\T$ which is of $\sigma  $-measure $1$, $\textrm{sp}[\ker(T-\lambda )\textrm{ ; } \lambda \in A]$ is dense in $X$.
\end{definition}
In other words if we take out from the unit circle a set of $\sigma  $-measure $0$ of \eva s, the \eve s associated to the remaining \eva s still span $X$.
\par\smallskip
The following result is proved in \cite{BG2}:

\begin{theorem}\cite{BG2}\label{th0}
If $T$ is a bounded \op\ acting on a \sep\ infinite dimensional complex Hilbert space $H$, and if $T$ has a \ps, then $T$ is \fhy.
\end{theorem}

As mentioned, above, the method of proof of this statement is rather complicated, since it involves the construction of an \inv\ \erg\ \ga\ \mea\ for the \op\ $T$. Moreover \ga\ \mea s are much easier to deal with on Hilbert spaces than on general Banach spaces, because a complete description of the covariance operators of Gaussian measures is available on Hilbert spaces. We refer the reader to \cite[Ch. 6, Section 2]{BL} for a study of Gaussian measures in the Hilbertian setting, and to \cite{VTC}
for a presentation in the Banach space case. This explains why, when trying to prove a Banach space version of Theorem \ref{th0}, we were compelled in \cite{BG3} to add some assumption concerning either the geometry of the space (that $X$ is of type $2$, for instance) or the regularity of the eigenvector fields of the \op\ (that they can be parametrized in a ``smooth'', i.e. $\alpha $-H\"olderian way for some suitable $\alpha $). See the book \cite{BM} for more details on these results.
\par\smallskip
Thus the following question remained open in \cite{BG3}:

\begin{question}\cite{BG3}\label{q1}
 If $X$ is a general \sep\ complex infinite-dimensional Banach space and $T$ is a bounded operator on $X$ which has a \ps, must $T$ be \fhy?
\end{question}

It is proved in \cite{BG2} that if $T$ has perfectly spanning unimodular \eve s, then $T$ must already be \hy.
The main result  of this paper is an affirmative answer to Question \ref{q1}:

\begin{theorem}\label{th1bis}
Let $T$ be a bounded \op\ acting on a complex Banach space $X$. If the \eve s of $T$ associated to \eva s of modulus $1$ are perfectly spanning, then $T$ is \fhy.
\end{theorem}

The proof of Theorem \ref{th1bis} is the object of the first three sections of the paper.
It relies on the construction of an explicit \inv\ \mea\ and on the use of Birkhoff's ergodic theorem, as in \cite{G4} where a ``Random Frequent Hypercyclicity Criterion'' is proved using somewhat similar tools. One interesting point is that this measure is constructed using independent Steinhaus variables, instead of \ga\ ones as in the previous constructions of \cite{BG3} and \cite{G4}. We obtain on our way (in Section 4) several characterizations, which are of interest in themselves, of \ops\ having perfectly spanning unimodular \eve s.
\par\smallskip
It is also interesting to note that the \op\ $T$ of Theorem \ref{th1bis} will never be \erg\ \wrt\ one of the \inv\ \mea s constructed in the proof: this result is proved in Section 5.
\par\smallskip
In the last section of the paper we collect miscellaneous remarks and open questions. In particular we mention how Theorem \ref{th1bis} can be applied to retrieve the main result of \cite{DFGP}, namely that any infinite-dimensional separable complex Banach space with an unconditional Schauder decomposition supports a frequently hypercyclic operator.
\par\bigskip

\textbf{Acknowledgement:} I wish to thank the referee for his/her very careful reading of the paper, and for his/her numerous comments which greatly improved the presentation of this work.

\section{Strategy for the proof of Theorem \ref{th1bis}}
We are going to derive Theorem \ref{th1bis} from our forthcoming Theorem \ref{th2}, which states that if $T$ is a bounded \hy\ \op\ on a separable infinite-dimensional complex Banach space $X$ whose \eve s associates to \eva s of modulus $1$ span a dense subspace of $X$, then $T$ is \fhy\ provided the unimodular \eve s of $T$ satisfy some additional assumption (H). Assumption (H) is a priori weaker than the assumption that $T$ has perfectly spanning unimodular \eve s, although it will turn out to be actually equivalent to it (see Section 4).
\par\smallskip
Before stating assumption (H), let us start with two elementary lemmas.
Let $T$ be a \hy\ \op\ on $X$ whose \eve s associated to unimodular \eva s span a dense subspace of $X$. We denote by $\sigma  _{p}(T)\cap\T$ the set of \eva s of $T$ of modulus $1$.
\begin{lemma}\label{lem1}
Let $F$ be a finite subset of $\sigma  _{p}(T)\cap\T$. Then ${\textrm{sp}}[\ker(T-\lambda ) \textrm{ ; } \lambda \in\T\setminus F]$ is dense in $X$.
\end{lemma}

\begin{proof}
Suppose that $X_{0}=\overline{\textrm{sp}}[\ker(T-\lambda ) \textrm{ ; } \lambda \in\T\setminus F]$ is not equal to $X$, and let $\overline{T}$ be the operator induced by 
$T$ on the quotient space $\overline{X}=X/X_{0}$. Then $\overline{T}$ is \hy\ on $\overline{X}$.
Let $(x_{n})_{n\geq 1}$ be a sequence of elements of $\bigcup_{\lambda \in\T\setminus F}
\ker(T-\lambda )$ which span $X_{0}$, and $(y_{n})_{n\geq 1}$ a sequence of elements of $\bigcup_{\lambda \in F}
\ker(T-\lambda )$ such that the set $\{x_{n}, \,y_{n} \textrm{ ; } n\geq 1\}$ span a dense subspace of $X$: then $\{\overline{x_{n}}, \,\overline{y_{n}} \textrm{ ; } n\geq 1\}$ span a dense subspace of $\overline{X}$, i.e. $\{\overline{y_{n}} \textrm{ ; } n\geq 1\}$ span a dense subspace of $\overline{X}$. Hence the \eve s associated to the \eva s of $\overline{T}$ belonging to the finite set $F$ span a dense subspace of $\overline{X}$, so that $\prod_{\lambda \in F}(\overline{T}-\lambda )=0$, which contradicts the hypercyclicity of $\overline{T}$. Hence $X_{0}=X$.
\end{proof}

The proof of Lemma \ref{lem1} actually shows:

\begin{lemma}\label{lem2}
Let $(x_{n})_{n\geq 1}$ be a sequence of \eve s of $T$, $Tx_{n}=\lambda _{n}x_{n}$, $|\lambda _{n}|=1$, such that ${\textrm{sp}}[x_{n}\textrm{ ; } n\geq 1]$ is dense in $X$. If $F$ is any finite subset of $\sigma  _{p}(T)\cap\T$, then ${\textrm{sp}}[x_{n}\textrm{ ; } n\in A_{F}]$ is dense in 
$X$,
where $A_{F}=\{n\geq 0 \textrm{ ; } \lambda _{n}\not \in F\}$.
\end{lemma}

Suppose now that $T$ satisfies the following assumption (H):
\par\medskip
\emph{
There exists a sequence $(x_{n})_{n\geq 1}$ of \eve s of $T$, $Tx_{n}=\lambda _{n}x_{n}$, $|\lambda _{n}|=1$, $\lambda _{n}=e^{2i\pi\theta _{n}}$, $\theta _{n}\in ]0,1]$, $||x_{n}||=1$, having the following properties:}
\begin{itemize}
\item[(1)] \emph{whenever $(\lambda _{n_{1}},\ldots, \lambda_{n_{k}})$ is a finite family of distinct elements of the set $\{\lambda _{n} \textrm{ ; } n\geq 1 \}$, the family $(\theta  _{n_{1}},\ldots, \theta _{n_{k}})$ consists of $\Q$-independent irrational numbers;}

\item[(2)] \emph{${\textrm{sp}}[x_{n}\textrm{ ; } n\geq 1]$ is dense in $X$;}

\item[(3)] \emph{for any finite subset $F$ of $\sigma  _{p}(T)\cap\T$ we have $\overline{\{x_{n} \textrm{ ; } n\geq 1\}}=\overline{\{x_{n} \textrm{ ; } n\in A_{F}\}}$, where $A_{F}=\{k\geq 0 \textrm{ ; } \lambda _{k}\not \in F\}$.}
\end{itemize}

\par\medskip
Assertion (3) of assumption (H) states that given any finite set $F$ of \eva s of $T$, any $x_{n}$ can be approximated as closely as we wish by \eve s associated to \eva s not belonging to $F$. 
Assertion (1) ensures that we have some ``independence'' of the \eva s $\lambda _{n}$; this will turn out to be necessary in the proof of Theorem \ref{th1bis}.
It is not difficult to see already (more details will be given in Section 4 later on) that assumption (H) will be satisfied provided the unimodular \eve s of $T$ can be parametrized via countably many continuous eigenvector fields. As will also be seen in Section 4, this seemingly weaker assumption is in fact equivalent to the requirement that the unimodular \eve s of $T$ be perfectly spanning.
\par\smallskip
The first step in the proof of Theorem \ref{th1bis} is to prove the following statement:

\begin{theorem}\label{th2}
If $T$ is a bounded \op\ on $X$ which is \hy\ and satisfies assumption (H), then $T$ is \fhy.
\end{theorem}

Let $(\Omega ,\mathcal{F}, \P)$ be a standard probability space, and $(\chi_{n})_{n\geq 1}$ a sequence of  independent \st\ variables on $(\Omega ,\mathcal{F}, \P)$:
$\chi_{n}:\Omega \longrightarrow \T$, and for any subarc $I$ of $\T$, $$\P(\chi_{n} \in I)=\frac{|I|}{2\pi},$$ where $|I|$ is the length of $I$. We have $\E(f(\chi_{n}))=\frac{1}{2\pi}
\int_{0}^{2\pi}f(e^{i\theta })d\theta $ for any continuous function $f$ on $\T$, so that
$\E(\chi_{n})=0$ and $\E|\chi_{n}|^{2}=1$ for any $n\geq 1$. One important feature of these \st\ variables is that for any unimodular numbers $\lambda _{n}$, $\lambda _{n}\chi_{n}$ and $\chi_{n}$ have the same law. This makes these variables quite useful for constructing invariant measures for linear operators.
\par\smallskip
Suppose that $(y_{n})_{n\geq 1}$ is a sequence of \eve s of $T$, $Ty_{n}=\lambda _{n}y_{n}$, $|\lambda _{n}|=1$, such that the random series
$$\Phi(\omega )=\sum_{n\geq 1}\chi_{n}(\omega )y_{n}$$ is convergent almost everywhere. 
Then it is possible  to define a \mea\ $m$ on the Banach space $X $ by setting for any Borel subset $A$ of $X$
$$m(A)=\P(\{\omega \in \Omega  \textrm{ ; } \sum_{n\geq 1}\chi_{n}(\omega )y_{n}\in A\}).$$
The \mea\ $m$ is invariant by $T$:
\begin{eqnarray*}
m(T^{-1}(A))&=&\P(\{\omega \in \Omega  \textrm{ ; } \sum_{n\geq 1}\chi_{n}(\omega )Ty_{n}\in A\})\\
&=&\P(\{\omega \in \Omega  \textrm{ ; } \sum_{n\geq 1}\chi_{n}(\omega )\lambda _{n}y_{n}\in A\}).
\end{eqnarray*}
Since $|\lambda _{n}|=1$, $\lambda _{n}\chi_{n}$ and $\chi_{n}$ have the same law, and thus $m(T^{-1}(A))=m(A)$.
\par\smallskip
Our strategy to prove Theorem \ref{th1bis} is to construct a sequence $(y_{n})_{n\geq 1}$ of unimodular \eve s of $T$ which is such that
\begin{itemize}
 \item[(a)] the associated random series $\Phi(\omega )$ converges a.e. on $\Omega $;

\item[(b)] for almost every $\omega \in\Omega $, $\Phi(\omega )$ is \hy\  for $T$.
\end{itemize}
Once the sequence $(y_{n})_{n\geq 1}$ satisfying (a) and (b) is constructed, it is not difficult to see that $\Phi(\omega ) $ is \fhy\ for $T$ for almost every $\omega \in \Omega $: this is proved in \cite[Prop. 3.1]{G4} under the assumption that the measure $m$ associated to $\Phi$ is \nd, i.e. that $m(U)>0$ for any non-empty open subset $U$ of $X$. This a priori assumption that $m$ be \nd\ is in fact not necessary:

\begin{proposition}\label{prop1}
 Suppose that there exists a \mea\ $m$ which is \inv\ by $T$ and such that $m(HC(T))=1$, where $HC(T)$ denotes the set of \hy\ vectors for $T$. Then the set $FHC(T)$ of \fhy\ vectors for $T$ also satisfies $m(FHC(T))=1$. In particular $T$ is \fhy.
\end{proposition}

\begin{proof}
For any non-empty
open subset $U$ of $X$, $$HC(T)\subseteq \bigcup_{n\geq 0}T^{-n}(U)$$
so that $m(\bigcup_{n\geq 0}T^{-n}(U))=1$. Since $m(U)=m(T^{-n}(U))$ for any $n\geq 1$, it is impossible that $m(U)=0$. So $m(U)>0$, and $m$ actually has full support. The rest of the proof then goes exactly as in \cite[Prop. 3.1]{G4}. We recall the argument for completeness's sake:
 since $m$ is $T$-invariant,
Birkhoff's theorem implies that for $m$-almost every $x$ in $X$,
$$\frac{1}{N}\#\{n\leq N \textrm{ ; } T^{n}x \in U\}\longrightarrow \E(\chi_{U}|\mathcal{I})(x),$$
where $\chi_{U}$ is the characteristic function of the set $U$ and $\mathcal{I}$
the $\sigma $-algebra of $T$-invariant subsets of $(X,\mathcal{B},m)$. Now $\E(\chi_{U}|\mathcal{I})$
is a $T$-invariant function which it is positive almost everywhere on the set
$\bigcup_{n\geq 0}T^{-n}(U)$, which has measure $1$. 
So $\E(\chi_{U}|\mathcal{I})$
is positive almost everywhere, and it follows that
$m$-almost every $x$ is \fhy\ for $T$. 
\end{proof}

In the works \cite{BG2}, \cite{BG3}, \cite{G4}, \inv\ \mea s were constructed using sums of \ind\ \ga\ variables $\sum g_{n}(\omega )x_{n}$, and taking advantage of the rotational invariance of the law of $g_{n}$. It is important here that we consider \st\ variables instead of \ga\ variables, as will be seen shortly.
\par\smallskip
Let us summarize: we are looking for  a sequence $(y_{n})_{n\geq 1}$ of \eve s of $T$,  such that $\Phi(\omega )=\sum_{n\geq 1}\chi_{n}(\omega )y_{n}$ defines an \inv\ \mea\ $m$ such that $m(HC(T))=1$. The construction of such a sequence $(y_{n})_{n\geq 1}$ will be done by induction, and by blocks: at step $k$ we construct the vectors $y_{n}$ for $n\in [s_{k-1}, s_{k}-1]$, where $(s_{k})$ is a certain fast increasing sequence of integers with $s_{0}=1$.
\par\smallskip
Before beginning the construction we state separately one obvious fact, which will be used repeatedly in the forthcoming proof:

\begin{lemma}\label{lem4}
 Let $a$ be a complex number, and $\varepsilon >0$. There exists a finite family $(a_{1},\ldots, a_{N})$ of complex numbers such that
\begin{itemize}
 \item [(i)] $a_{1}+\ldots+a_{N}=a$
\item[(ii)] $|a_{1}|^{2}+\ldots+|a_{N}|^{2}<\varepsilon .$
\end{itemize}
\end{lemma}

\begin{proof}
Just take $N$ so large that $\frac{a^{2}}{N}<\varepsilon $ and set $a_{i}=\frac{a}{N}$ for any $i=1,\ldots, N$.
\end{proof}

\section{Proof of Theorem \ref{th2}: frequent hypercyclicity of $T$ under assumption (H)}

Let $(U_{n})_{n\geq 1}$ be a countable basis of open subsets of $X$, and let $(x_{n})_{n\geq}$ be a sequence of \eve s of $T$, $||x_{n}||=1$, $Tx_{n}=\lambda _{n}x_{n}$, satisfying assumption (H).
\par\smallskip
\textbf{Step 1:} 
Since $T$ is \hy, there exists an integer $p_{1}$ such that $T^{p_{1}}(B(0,\frac{1}{2}))\cap U_{1}$ is non-empty. As the vectors $x_{k}$, $k\geq 1$, span a dense subspace of $X$, there exists a finite linear combination $u_{1}$ of the vectors $x_{k}$ such that $||u_{1}||<\frac{1}{2}$ and $T^{p_{1}}u_{1}\in U_{1}$. Let us write $u_{1}$ as
$$u_{1}=\sum_{k\in I_{1}}\alpha _{k}x_{k}$$ where $I_{1}=[1,r_{1}]$ is a certain finite interval of $[1,+\infty [$. Since the linear space $\textrm{sp}[x_{k}\textrm{ ; } k\in I_{1}]$ is finite-dimensional, there exists a positive constant $M_{1}$ such that for every family $(\beta _{k})_{k\in I_{1}}$ of complex numbers,
$$||\sum_{k\in I_{1}}\beta _{k}x_{k}||\leq M_{1}\left(\sum_{k\in I_{1}}|\beta _{k}|^{2}\right)^{\frac{1}{2}}.$$
Let $\delta _{1}$ be a very small positive number. By Lemma \ref{lem4}, we can write each $\alpha _{k}$, $k\in I_{1}$, as
$$\alpha _{k}=\sum_{j\in J_{k}^{1}}a_{j}^{(k)},$$ where the sets $J_{k}^{1}$, 
$k\in I_{1}$, are successive intervals of $[1,+\infty [$ and
$$\sum_{k \in I_{1}}\left( \sum_{j \in J_{k}^{1}}|a_{j}^{(k)}|^{2}\right)^{\frac{1}{2}}<\delta _{1}.$$
Thus $u_{1}$ can be rewritten as 
$$u_{1}=\sum_{k\in I_{1}}\left(\sum_{j\in J_{k}^{1}}a_{j}^{(k)} \right)x_{k}.$$
Let $\gamma _{1}$ be a very small positive number, to be chosen later on in the proof. Assumption (H) implies that there exists a family $x_{j}^{(k)}$, $k \in I_{1}$, $j\in J_{k}^{1} $, of elements of the set $\{x_{n} \textrm{ ; } n\geq 1\}$ such that for any $k\in I_{1}$ and $j\in J_{k}^{1}$,
$$||x_{k}-x_{j}^{(k)}||<\gamma _{1}$$ and the \eva s $\lambda _{j}^{(k)}$ associated to the eigenvectors $x_{j}^{(k)}$ are all distinct. Hence the arguments $\theta _{j}^{(k)}$ of the \eva s $\lambda _{j}^{(k)}=e^{2i\pi\theta _{j}^{(k)}}$ form a $\Q$-independent sequence of irrational numbers. Set
$$v_{1}=\sum_{k\in I_{1}}\sum_{j\in J_{k}^{1}}a_{j}^{(k)} x_{j}^{(k)}.$$ We have
$$||u_{1}-v_{1}||\leq \sum_{k\in I_{1}}\sum_{j\in J_{k}^{1}}|a_{j}^{(k)}|\, ||x_{j}^{(k)}-x_{k}|| \leq \gamma _{1}\, \sum_{k\in I_{1}}\sum_{j\in J_{k}^{1}}|a_{j}^{(k)}| $$ so that $||u_{1}-v_{1}|| $ can be made arbitrarily small if $\gamma _{1}$ is small enough. Hence taking $\gamma _{1}$ very small we can ensure that $T^{p_{1}}v_{1}$ belongs to $U_{1}$, i.e. that 
$$\sum_{k\in I_{1}}\sum_{j\in J_{k}^{1}}a_{j}^{(k)}(\lambda _{j}^{(k)})^{p_{1}}x_{j}^{(k)}\in U_{1}.$$
Let $(\chi_{j}^{(k)})_{k\in I_{1}, j\in J_{k}^{1}}$ be a family of \ind\ \st\ variables, and define on $(\Omega ,\mathcal{F},\P)$ a random function $\Phi_{1}$ by
$$\Phi_{1}(\omega )=\sum_{k\in I_{1}} \sum_{j\in J_{k}^{1}}\chi_{j}^{(k)}(\omega )\,a_{j}^{(k)}x_{j}^{(k)} .$$
Our aim is now to estimate the expectation $\E||\Phi_{1}(\omega )||$. In order to do this, let us consider the auxiliary random function
$$\Psi_{1}(\omega )=\sum_{k\in I_{1}}\left( \sum_{j\in J_{k}^{1}}\chi_{j}^{(k)}(\omega )\,a_{j}^{(k)}\right)x_{k}.$$ Writing 
$$\beta _{k}(\omega )=\sum_{j\in J_{k}^{1}}\chi_{j}^{(k)}(\omega )\,a_{j}^{(k)},$$ we have
$$||\Psi_{1}(\omega )||\leq M_{1}\, \left(\sum_{k\in I_{1}}|\beta _{k}(\omega )|^{2}\right)^{\frac{1}{2}}\leq M_{1}\,\sum_{k\in I_{1}}|\beta _{k}(\omega )|$$ so that 
$$\E||\Psi_{1}(\omega )||\leq M_{1}\,\sum_{k\in I_{1}}\E\left|\sum_{j\in J_{k}^{1}}\chi_{j}^{(k)}(\omega )\,a_{j}^{(k)}\right|.$$ Now the ``Steinhaus version''of Khinchine inequalities states that there is a universal constant $C>0$ such that for any sequence
$(a_{n})_{n\geq 1}$ of complex numbers, we have for any $N\geq 1$
$$\frac{1}{C}\left(\sum_{n=1}^{N}|a_{n}|^{2}\right)^{\frac{1}{2}}\leq
\E \left|\sum_{n=1}^{N}\chi_{n}(\omega )a_{n}\right|
\leq C\left(\sum_{n=1}^{N}|a_{n}|^{2}\right)^{\frac{1}{2}}.$$
Hence $$\E||\Psi_{1}(\omega )||\leq M_{1}\,C\,\sum_{k\in I_{1}}\left(\sum_{j\in J_{k}^{1}}|a_{j}^{(k)}|^{2}\right)^{\frac{1}{2}}<M_{1}\,C\,\delta _{1}.$$
Hence if $\delta _{1}$ is chosen very small \wrt\ $M_{1}$, we can ensure that $\E||\Psi_{1}(\omega )||<4^{-1}$ for instance. Now
$$||\Phi_{1}(\omega )-\Psi_{1}(\omega )||\leq
\sum_{k\in I_{1}}\sum_{j\in J_{k}^{1}}|a_{j}^{(k)}|\, ||x_{j}^{(k)}-x_{k}|| \leq \gamma _{1}
\sum_{k\in I_{1}}\sum_{j\in J_{k}^{1}}|a_{j}^{(k)}|. $$ Thus if $\gamma _{1}$ is small enough, $\E||\Phi_{1}(\omega )-\Psi_{1}(\omega )||$ is so small that $\E||\Phi_{1}(\omega )||<4^{-1}$ too (recall that $M_{1}$ is chosen first, then $\delta _{1}$ is chosen very small \wrt\ $M_{1}$, and lastly $\gamma _{1}$ is chosen very small \wrt\ $\delta _{1}$).
\par\smallskip
Our next goal is to show that there exists a finite family $\mathcal{P}_{1}$ of integers such that for almost every $\omega \in \Omega $, there exists an integer $p_{1}(\omega )\in \mathcal{P}_{1}$ such that $T^{p_{1}(\omega )}\Phi_{1}(\omega )$ belongs to $U_{1}$.
\par\smallskip
We have for any $p\geq 0$
$$T^{p}\Phi_{1}(\omega )=\sum_{k\in I_{1}} \sum_{j\in J_{k}^{1}}\chi_{j}^{(k)}(\omega )\,(\lambda _{j}^{(k)})^{p}\,a_{j}^{(k)}x_{j}^{(k)}.$$
Let $(\mu_{j}^{(k)})_{k\in I_{1}, j\in J_{k}^{1}}$ be any family of unimodular numbers indexed by the sets $I_{1}$ and $J_{k}^{1}$, $k\in I_{1}$. Since the arguments of the $\lambda _{j}^{(k)}$ are $\Q$-independent irrational numbers, there exists for any $\eta_{1}>0$
 an integer $p\geq 1$ such that for any $k\in I_{1}$ and any $j\in J_{k}^{1}$
$$|(\lambda _{j}^{(k)})^{p}-\mu_{j}^{(k)}|<\frac{\eta_{1}}{2}.$$ Considering a finite $\eta_{1}/2$-net of the set $\T^{\sum|J_{k}^{1}|}$, we obtain that there exists a finite family $\mathcal{Q}_{1}$ of integers such that for almost every $\omega \in \Omega $ there exists an integer $p(\omega )\in \mathcal{Q}_{1}$ such that  for any $k\in I_{1}$ and any $j\in J_{k}^{1}$,
$$|(\lambda _{j}^{(k)})^{p(\omega)}-\overline{\chi_{j}^{(k)}(\omega )}|<\eta_{1}.$$ Now if $\rho  _{1}$ is any positive number, there exists an $\eta_{1}>0$ such that if $p$ is such that
$|{\chi_{j}^{(k)}(\omega )}(\lambda _{j}^{(k)})^{p}-1|<\eta_{1}$ for any $k \in I_{1}$ and $j\in J_{k}^{1}$, then $||T^{p}\Phi_{1}(\omega )-v_{1}||<\rho  _{1}$. Indeed in this case
\begin{eqnarray*}
T^{p}\Phi_{1}(\omega )-v_{1}&=&||\sum_{k\in I_{1}} \sum_{j\in J_{k}^{1}}\left(\chi_{j}^{(k)}(\omega )\,(\lambda _{j}^{(k)})^{p}-1\right)\,a_{j}^{(k)}x_{j}^{(k)}||\\
&\leq&\eta_{1}\,\sum_{k\in I_{1}} \sum_{j\in J_{k}^{1}} |a_{j}^{(k)}|<\rho  _{1}
\end{eqnarray*}
if $\eta_{1}$ is sufficiently small \wrt\ $\rho  _{1}$.
 Choose $\rho  _{1}$ such that $$T^{p_{1}}v_{1}+B(0,\rho_{1}||T||^{p_{1}})\subseteq U_{1},$$ then $\eta_1$ as above, and take $\mathcal{P}_{1}=p_{1}+\mathcal{Q}_{1}$: for almost every $\omega \in \Omega $, there exists a $p(\omega )\in \mathcal{Q}_{1}$ such that $||T^{p(\omega )}\Phi_{1}(\omega )-v_{1}||<\rho  _{1}$. Thus $$||T^{p_{1}+p(\omega )}\Phi_{1}(\omega )-T^{p_{1}}v_{1}||<\rho  _{1}\, ||T||^{p_{1}}$$ so that $T^{p_{1}+p(\omega )}\Phi_{1}(\omega )$ belongs to $U_{1}$. 
\par\smallskip
Let us summarize what has been done in this first step: we have constructed a function $\Phi_{1}(\omega )$ which is a finite \st\ sum of \eve s of $T$ associated to distinct \eva s, such that 

$\bullet$ $\E||\Phi_{1}(\omega )||<4^{-1}$

$\bullet$ there exists  a finite set $\mathcal{P}_{1}$ of integers such that for almost every $\omega \in \Omega $, there exists an integer $p_{1}(\omega )\in \mathcal{P}_{1}$ such that $T^{p_{1}(\omega )}\Phi_{1}(\omega )$ belongs to $U_{1}$. Let $\pi_{1}$ denote the maximum of the set $\mathcal{P}_{1}$.
\par\medskip
\textbf{Step 2:}
Let $V_{2}$ be a non-empty open subset of $X$ and $\kappa _{2}$ be a positive number such that $V_{2}+B(0,2\kappa _{2})\subseteq U_{2}$.
For any $p\geq 0$ and almost every $\omega \in\Omega $ we have
$$T^{p}\Phi_{1}(\omega )-\Phi_{1}(\omega )=\sum_{k\in I_{1}} \sum_{j\in J_{k}^{1}}\chi_{j}^{(k)}(\omega )\,\left((\lambda _{j}^{(k)})^{p}-1\right)\,a_{j}^{(k)}x_{j}^{(k)}.$$
There exists $\eta_{2}>0$ such that if $p$ is in the set $D_{2}$ of integers such that 
$|(\lambda _{j}^{(k)})^{p}-1|<\eta_{2}$ for every $k\in I_{1}$ and every $j\in J_{k}^{1}$, then for almost every $\omega \in\Omega $
$$||T^{p}\Phi_{1}(\omega )-\Phi_{1}(\omega )||<\kappa _{2}.$$ Observe that this set $D_{2}$ has bounded gaps. Indeed there exists a set $D'_{2}$ of positive density such that for any $k\in I_{1}$ and any $j\in J_{k}^{1}$, and for any $p\in D'_{2}$,
$|(\lambda _{j}^{(k)})^{p}-1|<\eta_{2}/2$. Then for any pair $(p,p')$ of elements of $D'_{2}$ we have 
$$|(\lambda _{j}^{(k)})^{p-p'}-1|\leq |(\lambda _{j}^{(k)})^{p}-1|+|(\lambda _{j}^{(k)})^{p}-1|<\eta_{2}.$$ Thus $(D_{2}'-D_{2}')\cap\N$ is contained in $D_{2}$. Since $D'_{2}$ has positive lower density,
$(D_{2}'-D_{2}')\cap\N$ has bounded gaps by a result of \cite{ST}. Hence $D_{2}$ has bounded gaps too. Let $r_{2}$ be such that any interval of $\N$ of length strictly larger than $r_{2}$ contains an element of $D_{2}$.
\par\smallskip
Now consider the set $E_{2}=\{p\geq 0 \textrm{ ; } T^{p}(B(0,2^{-2}))\cap V_{2}\not =\varnothing\}$. Since $T$ is \hy, $E_{2}$ is non-empty. But we can actually say more about $E_{2}$: as $T$ is \hy\ and has spanning unimodular \eve s, $T$ satisfies the \HC\ by \cite{G6}. Hence for any $r\geq 1$, the \op\ $T_{r}$ which is a direct sum of $r$ copies of $T$ on the direct sum $X_{r}$ of $r$ copies of $X$ is \hy. In particular $T_{r_{2}+1}$ is topologically transitive, which implies that there exists an integer $p$ such that $T^{p}(B(0,2^{-2}))\cap V_{2}\not =\varnothing,\;
T^{p}(B(0,2^{-2}))\cap T^{-1}(V_{2})\not =\varnothing,\ldots,
T^{p}(B(0,2^{-2}))\cap T^{-r_{2}}(V_{2})\not =\varnothing$. In other words $p,p+1,\ldots, p+r_{2}$ belong to $E_{2}$. Hence $E_{2}\cap D_{2}$ is non-empty. Let $p_{2}\in E_{2}\cap D_{2}$:
$$||T^{p_{2}}\Phi_{1}(\omega )-\Phi_{1}(\omega )||<\kappa _{2}\quad \textrm{ for almost every } \omega \in\Omega ,$$ and 
$$T^{p_{2}}(B(0,2^{-2}))\cap V_{2}\not =\varnothing.$$
\par\smallskip
Let $F_{1}=\{\lambda _{j}^{(k)} \textrm{ ; } k\in I_{1}, j\in J_{k}^{1}\}$ be the set of \eva s which appear in Step 1 of the construction, and $A_{F_{1}}=\{k\geq 1 \textrm{ ; } \lambda _{k}\not \in F_{1}\}$. As $\textrm{sp}[x_{k} \textrm{ ; } k\in A_{F_{1}}]$ is dense in $X$, there exists a vector $u_{2}$ which is a finite linear combination of vectors $x_{k}$, $k\in A_{F_{1}}$, such that $T^{p_{2}}u_{2}\in V_{2}.$ We write $$u_{2}=\sum_{k\in I_{2}}\alpha _{k}x_{k},$$ where $I_{2}$ is a suitably chosen interval of $\N$. Let $M_{2}>0$ be such that 
for every family $(\beta _{k})_{k\in I_{2}}$ of complex numbers,
$$||\sum_{k\in I_{2}}\beta _{k}x_{k}||\leq M_{2}\left(\sum_{k\in I_{2}}|\beta _{k}|^{2}\right)^{\frac{1}{2}}.$$
Then as in Step 1 we decompose each $\alpha _{k}$, ${k\in I_{2}}$, as
$$\alpha _{k}=\sum_{j\in J_{k}^{2}}a_{j}^{(k)},$$ where 
$$\sum_{k \in I_{2}}\left( \sum_{j \in J_{k}^{(2)}}|a_{j}^{(k)}|^{2}\right)^{\frac{1}{2}}<\delta _{2}$$ and $\delta _{2}$ is a very small positive number, determined later on in the construction.
Thus
$$u_{2}=\sum_{k\in I_{2}}\left(\sum_{j\in J_{k}^{2}}a_{j}^{(k)} \right)x_{k}.$$
For any $\gamma _{2}>0$, there exists a family $x_{j}^{(k)}$, $k\in I_{2}$, 
$j\in J_{k}^{2}$ of elements of the set $\{x_{n} \textrm{ ; } n\geq 1\}$ such that $||x_{k}-x_{j}^{(k)}||<\gamma _{2}$ for any $k\in I_{2}$ and
$j\in J_{k}^{2}$ and the \eva s $\lambda _{j}^{(k)}$ associated to the \eve s $x_{j}^{(k)}$ are all distinct and distinct from the elements of $F_{1}$ (i.e. the \eva s involved at Step 1 of the construction). Hence all the arguments $\theta _{j}^{(k)}$ of the eigenvalues $\lambda _{j}^{(k)}=e^{2i\pi\theta _{j}^{(k)}}$, $k\in I_{1}$ and $j\in J_{k}^{1}$, $k\in I_{2}$ and $j\in J_{k}^{2}$, form a $\Q$-independent sequence of irrational numbers. Set
$$v_{2}=\sum_{k\in I_{2}}\left(\sum_{j\in J_{k}^{2}}a_{j}^{(k)} \right)x_{j}^{(k)}.$$ If $\gamma _{2}$ is small enough, we have $T^{p_{2}}v_{2}\in V_{2}$. 
Let $(\chi_{j}^{(k)})_{k\in I_{2}, j\in J_{k}^{2}}$ be a family of \ind\ \st\ variables which are \ind\ from the family 
$(\chi_{j}^{(k)})_{k\in I_{1}, j\in J_{k}^{1}}$, and set 
$$\Phi_{2}(\omega )=\sum_{k\in I_{2}} \sum_{j\in J_{k}^{2}}\chi_{j}^{(k)}(\omega )\,a_{j}^{(k)}x_{j}^{(k)} .$$ The same reasoning as in Step 1 shows that if we take first $\delta _{2}$ very small \wrt\ $M_{2}$, and then $\gamma _{2}$ very small \wrt\ $\delta _{2}$, we can ensure that $\E||\Phi_{2}(\omega )||$ is as small as we want, namely that $$\E||\Phi_{2}(\omega )||<\frac{4^{-2}}{||T||^{\pi_{1}}}\cdot$$
\par\smallskip
We are now going to show that there exists  a finite family $\mathcal{P}_{2}$ of integers such that for almost every $\omega \in \Omega $, there exists $p_{2}(\omega )\in \mathcal{P}_{2}$ such that $$T^{p_{2}(\omega )}(\Phi_{1}(\omega )+\Phi_{2}(\omega ))-\Phi_{1}(\omega )\in U_{2}.$$
Indeed for any $p\geq 0$ we have 
\begin{eqnarray*}
 T^{p}(\Phi_{1}(\omega )+\Phi_{2}(\omega ))-\Phi_{1}(\omega )-v_{2}&=&
\sum_{k\in I_{1}} \sum_{j\in J_{k}^{1}}\chi_{j}^{(k)}(\omega )\,\left((\lambda _{j}^{(k)})^{p}-1\right)\,a_{j}^{(k)}x_{j}^{(k)}\\
&+&
\sum_{k\in I_{2}} \sum_{j\in J_{k}^{2}}\left(\chi_{j}^{(k)}(\omega )\,(\lambda _{j}^{(k)})^{p}-1\right)\,a_{j}^{(k)}x_{j}^{(k)}.
\end{eqnarray*}
Let $\eta_{2}>0$. By the irrationality and the $\Q$-independence of the arguments of all the $\lambda _{j}^{(k)}$ involved in the expression above, there exists a finite family $\mathcal{Q}_{2}$ of integers such that for almost every $\omega \in\Omega $ there exists an integer $p(\omega )\in \mathcal{Q}_{2}$ such that  

-- for every
$k\in I_{1}$ and $j\in J_{k}^{1}$, $|(\lambda _{j}^{(k)})^{p(\omega )}-1|<\eta_{2}$,

and 

-- for every $k\in I_{2}$ and $j\in J_{k}^{2}$, $|(\lambda _{j}^{(k)})^{p(\omega )}-\overline{\chi_{j}^{(k)}(\omega )}|<\eta_{2}$ .

Thus if $\eta_{2}$ is small enough, 
$$||T^{p(\omega )}(\Phi_{1}(\omega )+\Phi_{2}(\omega ))-\Phi_{1}(\omega )-v_{2}||<\frac{\kappa_{2}}{||T||^{p_{2}}}\cdot$$ Then
$$||T^{p(\omega )+p_{2}}(\Phi_{1}(\omega )+\Phi_{2}(\omega ))-T^{p_{2}}\Phi_{1}(\omega )-T^{p_{2}}v_{2}||<{\kappa_{2}}.$$ But 
$$||T^{p_{2}}\Phi_{1}(\omega )-\Phi_{1}(\omega )||<\kappa _{2}, $$ so that
$$||T^{p(\omega )+p_{2}}(\Phi_{1}(\omega )+\Phi_{2}(\omega ))-\Phi_{1}(\omega )-v_{2}||<{2\kappa_{2}}.$$
Hence if $\mathcal{P}_{2}=p_{2}+\mathcal{Q}_{2}$, using the fact that $T^{p_{2}}v_{2}\in V_{2}$ and $V_{2}+B(0,2\kappa _{2})\subseteq U_{2}$, we get that for almost every $\omega \in\Omega $ there exists $p_{2}(\omega )\in \mathcal{P}_{2}$ such that 
$$T^{p_{2}}(\Phi_{1}(\omega )+\Phi_{2}(\omega ))-\Phi_{1}(\omega )\in U_{2}.$$ Let $\pi_{2}$ denote the maximum of the set $\mathcal{P}_{2}$.
\par\medskip
\textbf{Step n:}
Continuing in this way, we construct at step $n$ a random \st\ function
$$\Phi_{n}(\omega )=\sum_{k\in I_{n}} \sum_{j\in J_{k}^{n}}\chi_{j}^{(k)}(\omega )\,a_{j}^{(k)}x_{j}^{(k)} $$ such that

$\bullet$ we have $$\E||\Phi_{n}(\omega )||<\frac{4^{-n}}{||T||^{\max(\pi_{1},\ldots, \pi_{n-1})}} \quad \textrm{ in particular }\quad  \E||\Phi_{n}(\omega )||<4^{-n}$$
\par\smallskip
$\bullet$ there exists  a finite family $\mathcal{P}_{n}$ of integers such that for almost every $\omega \in \Omega $, there exists $p_{n}(\omega )\in \mathcal{P}_{n}$ such that $$T^{p_{n}(\omega )}\left(\Phi_{1}(\omega )+\Phi_{2}(\omega )+\ldots+\Phi_{n}(\omega )\right)-\left(\Phi_{1}(\omega )+\ldots+\Phi_{n-1}(\omega )\right)\in U_{n}.$$ We denote by $\pi_{n}$ the maximum of the set $\mathcal{P}_{n}$.
\par\smallskip
All the \st\ variables $\chi_{j}^{(k)}$, $k\in I_{m}$, $j\in J_{k}^{m}$ with $m\leq n$ are independent, and the numbers $p_{n}(\omega )$ depend only on the construction until step $n$. In other words, $p_{n}$ is $\mathcal{F}_{n}$-measurable, where $\mathcal{F}_{n}$ denotes the $\sigma $-algebra generated by the variables $\chi_{j}^{(k)}$, $k\in I_{m}$, $j\in J_{k}^{m}$,  $m\leq n$.

\par\medskip
\textbf{Construction of the invariant measure:} We are now ready to construct our function $\Phi$. Set
$$\Phi(\omega )=\sum_{n\geq 1} \Phi_{n}(\omega )=\sum_{n\geq 1} \left(\sum_{k\in I_{n}} \sum_{j\in J_{k}^{n}}\chi_{j}^{(k)}(\omega )\,a_{j}^{(k)}x_{j}^{(k)}\right)$$
Since 
$$\E||\Phi(\omega )||\leq \sum_{n\geq 1} \E||\Phi_{n}(\omega )||\leq \sum_{n\geq 1} 4^{-n}<+\infty ,$$ the series of \st\ variables written above has a subsequence of partial sums which converges in $L^{1}(\Omega , \mathcal{F},\P ; X)$, and hence by L\'evy's inequalities the series defining $\Phi$ converges almost everywhere.
\par\smallskip
Recall that if we define $m$ by $m(A)=\P(\Phi\in A)$ for any Borel subset $A$ of $X$, $m$ is $T$-\inv\ since all the vectors $x_{j}^{(k)}$ are unimodular \eve s for $T$. We are going to show that $\Phi(\omega )$ is \hy\ for $T$ for almost every $\omega \in \Omega $, and this will conclude the proof of Theorem \ref{th2}.
\par\smallskip
For almost every $\omega \in \Omega $ we can write for every $n\geq 1$
\begin{eqnarray*}
T^{p_{n}(\omega )}\Phi(\omega )-\Phi(\omega )&=&
\left(T^{p_{n}(\omega )}\left(\sum_{m\leq n}\Phi_{m}(\omega )\right)-
\sum_{m< n}\Phi_{m}(\omega )\right)\\
&+&T^{p_{n}(\omega )}\left(\sum_{m> n}\Phi_{m}(\omega )\right)-
\sum_{m\geq n}\Phi_{m}(\omega ). 
\end{eqnarray*}
We know that for almost every $\omega \in \Omega $, the first term in this expression belongs to $U_{n}$. So we have to estimate the second and third terms. Let us begin with the third one:
$$\E||\sum_{m\geq n}\Phi_{m}(\omega )||\leq \sum_{m\geq n}4^{-m}=\frac{4}{3}4^{-n}.$$ By Markov's inequality
$$\P\left( ||\sum_{m\geq n}\Phi_{m}(\omega )||>2^{-n}\right)\leq \frac{4}{3}\,2^{-n},\quad \textrm{i.e.}\quad \P\left( ||\sum_{m\geq n}\Phi_{m}(\omega )||\leq 2^{-n}\right)\geq 1-\frac{4}{3}\,2^{-n}.$$
Hence the third term in the display above is small with large probability. As for the second term, we have

\begin{eqnarray*}
\E\left|\left|\sum_{m> n} T^{p_{n}(\omega )}\Phi_{m}(\omega )\right|\right|&\leq&
\sum_{m> n} \E\left|\left|T^{p_{n}(\omega )}\Phi_{m}(\omega )\right|\right|
\leq \sum_{m> n} \E(||T||^{p_{n}(\omega )}.||\Phi_{m}(\omega )||)\\
&\leq& \sum_{m> n} ||T||^{\pi_{n}}\E||\Phi_{m}(\omega )||
\end{eqnarray*}
since $\pi_{n}=\sup\{p_{n}(\omega ) \textrm{ ; } \omega \in \Omega \}$. Now since $m\geq n+1$,
$$\E||\Phi_{m}(\omega )||\leq \frac{4^{-m}}{||T||^{\max(\pi_{1},\ldots, \pi_{m-1})}}\leq
\frac{4^{-m}}{||T||^{\pi_{n}}}$$
so that
$$\E\left|\left|\sum_{m> n} T^{p_{n}(\omega )}\Phi_{m}(\omega )\right|\right|\leq \sum_{m>n} 4^{-m}=\frac{1}{3}
4^{-n}.$$
Thus
$$\P\left( ||\sum_{m> n}T^{p_{n}(\omega )}\Phi_{m}(\omega )||\leq 2^{-n}\right)\geq 1-\frac{1}{3}\,2^{-n}.$$
Putting everything together yields that for every $n\geq 1$,
$$\P\left( T^{p_{n}(\omega )}\Phi(\omega )-\Phi(\omega )\in U_{n}+B(0,2^{-(n-1)})\right)\geq 1-\frac{5}{3}\,2^{-n}.$$
We are now done: let $U$ be any non-empty open subset of $X$, and $(n_{l})_{l\geq 1}$ a sequence of integers such that $U_{n_{l}}+B(0,2^{-(n_{l}-1)})\subseteq U$. Let 
$A_{n_{l}}=\{ \omega \in \Omega  \textrm{ ; } T^{p_{n_{l}(\omega )}}\Phi(\omega )-\Phi(\omega ) \in U \}$: we have seen that $\P(A_{n_{l}})\geq 1-\frac{5}{3}\,2^{-n_{l}}$. 
If
$$A=\{\omega \in \Omega  \textrm{ ; there exists } l\geq 1
\textrm{ such that } T^{p_{n_{l}}(\omega )}\Phi(\omega )-\Phi(\omega )\in U\}=\bigcup_{l\geq 1}A_{n_{l}},$$ then $\P(A)\geq \sup_{l\geq 1}\P(A_{n_{l}})$ and thus $\P(A)=1$.
This being true for any non empty open subset of $X$, by considering a countable basis of open subsets of $X$ we obtain that for almost every $\omega \in\Omega $ the set 
$\{T^{p}\Phi(\omega )-\Phi(\omega ) \textrm{ ; } p\geq 1\}$ is dense in $X$. This means that $\Phi(\omega )$ is \hy\ for almost every $\omega \in\Omega $, and this concludes the proof of Theorem \ref{th2}.

\par\medskip
\begin{remark}
 Suppose that $X$ is a Hilbert space, and  for $n\geq 1$ and $k\in I_n$, $j\in J_{k}^{n}$, denote by $y_{j}^{(k)}$ the vector $y_{j}^{(k)}=a_{j}^{(k)}x_{j}^{(k)}$. Then $\sum_{n\geq 1}\sum_{k\in I_n}\sum_{j\in J_k^n}
||y_{j}^{(k)}||^2$ is finite, and the proof above shows that the set of finite linear combinations $\sum_n\sum_{k\in I_n}\sum_{j\in J_k^n}
c_{j}^{(k)}y_{j}^{(k)}$ where $|c_{j}^{(k)}|=1$ is dense in $X$. This can be related to the following result of \cite{Ba}, which gives conditions on a sequence $(x_n)$ of vectors implying that the set of its linear combinations with unimodular coefficients is dense in $X$: if 
$\sum||x_n||^2$ is finite and $\sum|\pss{x}{x_n}|=+\infty$ for any non-zero $x$ in $X$, then
$\{\sum c_n
x_n\textrm{ ; }  |c_n|=1\}$ is dense in $X$. See \cite{BM} for an elegant proof of this fact. The simplest way to construct such a sequence $(x_n)$ is to take $x_n=\frac{1}{n}x_0$ with $x_0\not =0$ for a large number of $n$, let us say $n< n_1$, then
$x_n=\frac{1}{n}x_{n_1}$ for a large number of $n$'s with another suitable $x_{n_1}$, etc... A look at the proof of Theorem \ref{th2} shows that this is exactly what we do there: we ``duplicate'' each vector $x_k$ in a family of \eve s $x_{j}^{(k)}$, $j\in J_k^n$, associated to \eva s which are very close to the initial one but all distinct, with multiplicative coefficients $a_{j}^{(k)}$, and $\sum_{j\in J_k} |a_{j}^{(k)}|^2$ small but $\sum_{j\in J_k} |a_{j}^{(k)}|$ large.
\end{remark}

\section{Proof of Theorem \ref{th1bis}: frequent hypercyclicity of operators with perfectly spanning unimodular eigenvectors}

In order to prove Theorem \ref{th1bis}, it remains to show that assumption (H) is satisfied when the unimodular \eve s of $T$ are perfectly spanning. We are going to show that this follows from the (seemingly) weaker assumption that whenever $D$ is a countable subset of $\T$,
$\textrm{sp}[\ker(T-\lambda ) \textrm{ ; } \lambda \in \T \setminus D]$ is dense in $X$. This assumption comes from the pioneering work of Flytzanis \cite{Fl}, where the \erg\ theory of bounded \ops\ on Hilbert spaces was first studied. We prove that this condition is equivalent to the property that $T$ has perfectly spanning unimodular \eve s, and even to the stronger property that the unimodular \eve s of $T$ can be parametrized via countably many continuous eigenvector fields. In the statement of Theorem \ref{prop0}, 
$S_{X}=\{x\in X \textrm{ ; } ||x||=1\}$ denotes the unit sphere of $X$.

\begin{theorem}\label{prop0}
 Suppose that $T$ is a bounded \op\ on $X$. The following assertions are equivalent:
\begin{itemize}
 \item[(1)] for any countable subset $D$ of $\T$,
$\textrm{sp}[\ker(T-\lambda ) \textrm{ ; } \lambda \in \T \setminus D]$ is dense in $X$;
 
 \item[(2)] $T$ has perfectly spanning unimodular \eve s;
 
 \item[(3)] there exists a sequence $(K_{i})_{i\geq 1}$ of subsets of $\T$ which are homeomorphic to the Cantor set $2^{\omega }$ and a sequence $(E_{i})_{i\geq 1}$ of continuous functions $E_{i}:K_{i}\longrightarrow S_{X}$ such that:
 \begin{itemize}
 \item for any $i\geq 1$ and any $\lambda \in K_{i}$, $TE_{i}(\lambda )=\lambda E_{i}(\lambda )$;
 
 \item $\textrm{sp}[E_{i}(\lambda ) \textrm{ ; } i\geq 1, \lambda \in K_{i}]$ is dense in $X$.
\end{itemize}
\end{itemize}
\end{theorem}

Assuming for the moment that Theorem \ref{prop0} is proved, let us deduce Theorem \ref{th1bis} from it.

\begin{proof}[Proof of Theorem \ref{th1bis}]
As $T$ has perfectly spanning unimodular vectors, assertion (3) of Theorem \ref{prop0} above is satisfied. Since for each $i\geq 1$ the set $K_{i}$ is a Cantor-like subset of $\T$, we can construct a family of sequences of unimodular numbers $(\lambda _{n}^{(i)})_{n\geq 1}$, $i\geq 1$, which have the following properties:
 
-- for each $i\geq 1$, the set $\{\lambda _{n}^{(i)} \textrm{ ; } n\geq 1\}$ is dense in $K_{i}$;
 
-- all the numbers $\lambda _{n}^{(i)}$, $i,n\geq 1$, are distinct;
 
-- for any finite family $(\lambda _{n_{1}}^{(i_{1})},\ldots, \lambda _{n_{r}}^{(i_{r})})$ consisting of distinct elements, the arguments of these unimodular numbers consist of $\Q$-independent irrational numbers.

\par\smallskip
 
For each $i,n\geq 1$, let $x_{n}^{(i)}=E_{i}(\lambda _{n}^{(i)})$ denote the associated \eve\ via the \eve\ field $E_{i}$. 
\par\smallskip
It is now not difficult to see that the family $\{x_{n}^{(i)} \textrm{ ; }i,n\geq 1\}$ satisfies the requirements of assumption (H). First of all assertion (1) is true by construction. Then for any $i\geq 1$, any  $\lambda \in K_{i}$ can be written as a limit of a sequence of elements of the set $\{\lambda _{n}^{(i)} \textrm{ ; } n\geq 1\}$. The continuity of the function $E_{i}$ then implies that $E_{i}(\lambda )$ can be written as a limit of a sequence of elements of the set of vectors $\{x_{n}^{(i)} \textrm{ ; } n\geq 1\}$. Since the vectors $E_{i}(\lambda ) \textrm{ ; } i\geq 1, \lambda \in K_{i}$, span a dense subspace of $X$, it follows that  $\textrm{sp}[x_{n}^{(i)} \textrm{ ; } i,n\geq 1]$ is dense in $X$, hence that assertion (2) of assumption (H) is satisfied. Assertion (3) is again a consequence of the continuity of the \eve\ fields $E_{i}$: for any $i,n\geq 1$, $\lambda _{n}^{(i)}$ can be written as the limit of a sequence of distinct elements $(\lambda _{n_{k}}^{(i)})_{k\geq 1}$, which can in particular be chosen so as to avoid a given finite subset $F$ of $\T$. Then $x_{n}^{(i)}$ is the limit of the sequence $(x _{n_{k}}^{(i)})_{k\geq 1}$.
\par\smallskip
So $T$ satisfies assumption (H). Since $T$ is \hy\ \cite{BG2}, it follows from Theorem \ref{th2} that $T$ is \fhy, and thus Theorem \ref{th1bis} is proved.
\end{proof}

It remains now to prove Theorem \ref{prop0}.

\begin{proof}[Proof of Theorem \ref{prop0}]
$\textrm{(3)} \implies \textrm{(2)}$ is easy: for any $i\geq 1$ let $\sigma  _{i}$ be a continuous probability \mea\ supported on the compact set $K_{i}$, and let $\sigma  $ be the probability \mea\ $\sigma  =\sum_{i\geq 1} 2^{-i} \sigma  _{i}$. Then $\sigma  $ is continuous on $\T$. If $B$ is any $\sigma  $-measurable subset of $\T$ such that $\sigma  (B)=1$, then $\sigma  _{i}(B)=1$ for any $i\geq 1$. Suppose now that $x^{*}\in X^{*}$ is a functional which vanishes on $E_{i}(\lambda )$ for any $i\geq 1$ and any $\lambda \in K_{i}\cap B$: since $E_{i}$ is continuous on $K_{i}$, $\pss{x^{*}}{E_{i}(\lambda )}=0$ for any $\lambda \in K_{i}$, and thus $x^{*}=0$. Hence the \eve\ fields $E_{i}$, $i\geq 1$, are perfectly spanning \wrt\ $\sigma  $.
\par\smallskip
$\textrm{(2)} \implies \textrm{(1)}$ is also clear: if the unimodular \eve s of $T$ are perfectly spanning \wrt\ a continuous \mea\ $\sigma  $ on $\T$, then $\sigma  (D)=0$ for any countable subset $D$ of $\T$, so that (1) holds true.
\par\smallskip
$\textrm{(1)} \implies \textrm{(3)}$ is  the core of the proof of Theorem \ref{prop0}.
Let
$$A=S_{X}\cap \left(\bigcup_{\lambda \in\T} \ker(T-\lambda )\right)$$  be the set of \eve s of $T$ of norm $1$ associated to unimodular \eva s. Since 
$A$ is \sep, there exists a countable basis $(\Omega _{n})_{n\geq 1}$ of open subsets of $A$:
$\Omega _{n}=A\cap U_{n}$, where $U_{n}$ is open in $X$. Consider the set
$E$ of integers $n\geq 1$ having the following property: the set of eigenvalues $\lambda$  
such that $\Omega _{n}$ contains an element of $S_{X}\cap \ker(T-\lambda )$  is at most countable. Then let $\Delta  $ be the set of \eva s of $T$ such that there exists an $n\in E$ such that
$S_{X}\cap \ker(T-\lambda )\cap\Omega _{n}$ is non-empty. In other words $\lambda $ belongs to $\Delta  $ if and only if there is an \eve\ associated to $\lambda $  belonging to an $\Omega _{n}$ containing only \eve s associated to a countable family of \eva s:
$$\Delta  =\bigcup_{n\in E}\{\lambda \in \T \textrm{ ; } S_{X}\cap\ker(T-\lambda )\cap\Omega _{n}\not = \varnothing\}.$$
By the definition of $E$, 
$\Delta  $ is at most countable. Let $\lambda \in\T\setminus \Delta  $ be an \eva\ of $T$,  and let $x$ be an associated \eve\ of norm $1$. Let $V$ be an open neighborhood of $x$ in $A$, and let $p\geq 1$ be such that $\Omega _{p}\subseteq V $ and $x\in \Omega _{p}$. It is impossible that $p\in E$: if $p\in E$, then $x\in\ker(T-\lambda )\cap S_{X}\cap \Omega _{p}$ which is hence non-empty, and thus $\lambda$ belongs to $\Delta  $, which is contrary to our assumption. Hence $\Omega _{p}$, and so $V$, contain \eve s  of norm $1$ associated to an uncountable family of unimodular \eva s. Let us summarize this as follows: the set $$\Omega =S_{X}\cap \left(\bigcup_{\lambda \in \T\setminus \Delta  }\ker(T-\lambda )\right)$$
consists of \eve s of $T$ of norm $1$ such that any neighborhood of a vector $x\in \Omega $ contains \eve s of norm $1$ associated to uncountably many \eva s, in particular \eve s of norm $1$ associated to uncountably many \eva s not belonging to $\Delta   $.
\par\smallskip
Since $\Delta  $ is countable, $\textrm{sp}[\ker(T-\lambda ) \textrm{ ; } \lambda \in\T\setminus \Delta  ]=\textrm{sp}[\Omega ]$ is dense in $X$ by assumption (1). We choose  a sequence  $(u_{i})_{i\geq 1}$ of vectors of $\Omega $ which is dense in $\Omega $ and which is such that the corresponding \eva s $\lambda _{i}, i\geq 1$ are all distinct and belong to $\T\setminus \Delta  $. In particular the vectors $u_i$ span a dense subspace of $X$. Let us now fix $i\geq 1$ and construct $K_{i}$ and $E_{i}$. Let $s=(s_{1},\ldots, s_{n})\in 2^{<\omega }$ be a finite sequence of $0$'s and $1$'s. We associate to each such $s\in 2^{<\omega }$ an \eva\ $\lambda _{s}\in\{\lambda _{j}\textrm{ ; } j\geq 1\}$ and an \eve\ $u_{s}\in\{u_{j}\textrm{ ; } j\geq 1\}$ with $Tu_{s}=\lambda _{s}u_{s}$ in the following way:
\par\smallskip
$\bullet $ Step $1$: we start from $u_{(0)}=u_{i}$ and $\lambda _{(0)}=\lambda _{i}$. Let $n\not =i$ be such that
$||u_{n}-u_{(0)}||<2^{-1}$ and $|\lambda _{n}-\lambda _{(0)}|<2^{-1}$ (remark that if $||u_{n}-u_{(0)}||<2^{-1}$ is very small, $|\lambda _{n}-\lambda _{(0)}|<2^{-1}$ is automatically very small too). In particular $\lambda _{n}\not =\lambda _{(0)}$. We set $u_{(1)}=u_{n}$ and
$\lambda _{(1)}=\lambda _{n}$.
\par\smallskip
$\bullet $ Step $2$: we take $u_{(0,0)}=u_{(0)}$, $\lambda _{(0,0)}=\lambda _{(0)}$, and then take
$u_{(0,1)}$ in the set $\{u _{j}\textrm{  ; } j\geq 1\}$ and
$\lambda_{(0,1)}$ in the set $\{\lambda _{j}\textrm{  ; } j\geq 1\}$ so that $$||u _{(0,0)}-u_{(0,1)}||<2^{-1}||u_{(0)}-u_{(1)}||<2^{-2} \textrm{ and }|\lambda  _{(0,0)}-\lambda _{(0,1)}|<2^{-1}|\lambda _{(0)}-\lambda  _{(1)}|<2^{-2},$$ 
with $\lambda _{(0,1)}\not =\lambda _{(0,0)}$.
In the same way we take $u_{(1,0)}=u_{(1)}$ and $\lambda _{(1,0)}=\lambda _{(1)}$
and then choose $u_{(1,1)}$ and $\lambda _{(1,1)}$ very close to $u_{(1,0)}$  and $\lambda _{(1,0)}$ respectively so that  $$||u _{(1,0)}-u_{(1,1)}||<2^{-1}||u_{(0)}-u _{(1)}||<2^{-2} \textrm{ and } |\lambda _{(1,0)}-\lambda _{(1,1)}|<2^{-1}|\lambda _{(0)}-\lambda _{(1)}|<2^{-2},$$ with $\lambda _{(1,1)}$ not belonging to the set $\{\lambda _{(0,0)}, \lambda _{(0,1)}, \lambda _{(1,0)}\}$.
\par\smallskip
$\bullet $ Step $n$: we take
$u_{(s_{1},\ldots, s_{n-1},0)}=u _{(s_{1},\ldots, s_{n-1})}$ and  $\lambda _{(s_{1},\ldots, s_{n-1},0)}=\lambda _{(s_{1},\ldots, s_{n-1})}$, and then $u _{(s_{1},\ldots, s_{n-1},1)}$ very close to $u _{(s_{1},\ldots, s_{n-1},0)}$ and
$\lambda _{(s_{1},\ldots, s_{n-1},1)}$ very close to $\lambda _{(s_{1},\ldots, s_{n-1},0)}$,
so that
$$||u_{(s_{1},\ldots, s_{n-1},0)}-u _{(s_{1},\ldots, s_{n-1},1)}||<2^{-1}||u _{(s_{1},\ldots, s_{n-2},0)}-u _{(s_{1},\ldots, s_{n-2},1)}||<2^{-n}$$ and
$$|\lambda _{(s_{1},\ldots, s_{n-1},0)}-\lambda _{(s_{1},\ldots, s_{n-1},1)}|<2^{-1}|\lambda _{(s_{1},\ldots, s_{n-2},0)}-\lambda _{(s_{1},\ldots, s_{n-2},1)}|<2^{-n}.$$ We manage the construction in such a way that for all finite sequences $(s_{1},\ldots, s_{n})$ of $2^{\omega }$ of length $n$, the numbers $\lambda _{(s_{1},\ldots, s_{n})}$ are distinct.
\par\smallskip
This defines $\lambda _{s} $ and $u_{s}$ for $s\in 2^{<\omega }$. If now $s=(s_{1}, s_{2},\ldots)\in 2^{\omega }$ is an infinite sequence of $0$'s and $1$'s, we define $\lambda _{s}=\lim_{n\rightarrow+\infty }\lambda _{s_{|n}}$ and $u_{s}=\lim_{n\rightarrow+\infty }u_{s_{|n}}$, where $s_{|n}=(s_{1},\ldots, s_{n})$. These two limits do exist: indeed we have for any 
$n\geq 1$ that $|\lambda _{s_{|n-1}}-\lambda _{s_{|n}}|<2^{-n}$ and $||u_{s_{|n-1}}-u_{s_{|n}}||<2^{-n}$.
\par\smallskip
Let $\phi_{i}:2^{\omega }\rightarrow \T$ be the map defined by $\phi_{i}(s)=\lambda _{s}$. It is continuous and injective: if $s\not=s'$ are two distinct elements of $2^{\omega }$, and $p$ is the smallest integer such that $s_{n}\not=s'_{n}$ for any $n<p$, then for any $n\geq p$ we have
\begin{eqnarray*}
 |\lambda _{(s_{1},\ldots, s_{n})}-\lambda _{(s'_{1},\ldots, s'_{n})}|&\geq& |\lambda _{(s_{1},\ldots, s_{p-1}, s_{p})}-\lambda _{(s_{1},\ldots, s_{p-1}, s'_{p})}|\\&-&\sum_{m=p+1}^{n}|\lambda _{(s_{1},\ldots, s_{m})}-\lambda _{(s_{1},\ldots, s_{m-1})}|\\&-&\sum_{m=p+1}^{n}|\lambda _{(s'_{1},\ldots, s'_{m})}-\lambda _{(s'_{1},\ldots, s'_{m-1})}|\\
&\geq&
 |\lambda _{(s_{1},\ldots, s_{p-1}, s_{p})}-\lambda _{(s_{1},\ldots, s_{p-1}, s'_{p})}|\\&-&\left(\sum_{m=p+1}^{n}2^{-(m-p)}\right)
 |\lambda _{(s_{1},\ldots, s_{p-1}, s_{p})}-\lambda _{(s_{1},\ldots, s_{p-1}, s'_{p})}|\\
 &\geq&
 2^{-1} |\lambda _{(s_{1},\ldots, s_{p-1}, s_{p})}-\lambda _{(s_{1},\ldots, s_{p-1}, s'_{p})}|=\delta _{p}>0.
\end{eqnarray*}
It follows that $|\lambda _{s}-\lambda _{s'}|\geq\delta _{p}>0$, hence that $\lambda _{s}\not=\lambda_{s'}$, and $\phi_{i}$ is injective. We set $K_{i}=\phi_{i}(2^{\omega })$, and with this definition $K_{i}$ is a compact set homeomorphic to the Cantor set $2^{\omega }$ via the map $\phi_{i}$. Let now $E_{i}:K_{i}\rightarrow X$ be defined by $E_{i}(\lambda _{s})=u_{s}$: $E_{i}$ can be written as $E_{i}=\psi_{i}\circ\phi_{i}^{-1}$, where $\psi_{i}:2^{\omega }\rightarrow X$, $\psi_{i}(s)=u_{s}$. By the same argument as above $\psi_{i}$ is continuous on $2^{\omega }$, and since $\phi_{i}$ is an homeomorphism from $2^{\omega }$ onto $K_{i}$, $E_{i}$ is a continuous map from $K_{i}$ into $X$. Lastly $\phi_{i}({(0,0,\ldots)})=\lambda _{i}$ belongs to $K_{i}$,
and $E_{i}(\lambda _{i})=u_{i}$ so that $\textrm{sp}[E_{i}(\lambda ) \textrm{ ; } i\geq 1, \lambda \in K_{i}]$ is dense in $X$. Thus assertion (3)  is satisfied, and this finishes the proof of Theorem \ref{prop0}.
\end{proof}

The proof of Theorem \ref{prop0} actually yields the following result, which gives a rather weak condition for an operator to have perfectly spanning unimodular \eve s:

\begin{theorem}\label{prop0bis}
 Let $X$ be a complex \sep\ infinite-dimensional Banach space, and let $T$ be a bounded \op\ on $X$. Suppose that there exists a sequence $(u_{i})_{i\geq 1}$ of vectors of $X$ having the following properties:
 \begin{itemize}
  \item[(a)] for each $i\geq 1$, $u_{i}$ is an \eve\ of $T$ associated to an \eva\ $\lambda _{i}$ of $T$ with $|\lambda _{i}|=1$ and the $\lambda _{i}$'s all distinct;
  
  \item[(b)] $\textrm{sp}[u_{i}\textrm{ ; } i \geq 1]$ is dense in $X$;
  
  \item[(c)] for any $i\geq 1$ and any $\varepsilon >0$, there exists an $n\not =i $ such that $||u_{n}-u_{i}||<\varepsilon $.
 \end{itemize}
Then $T$ has a perfectly spanning set of unimodular \eve s, and hence $T$ is \fhy.
\end{theorem}

In particular $T$ is \fhy\ as soon as the following assumption (H') holds true:
\par\medskip
\emph{There exists a sequence $(x_{n})_{n\geq 1}$ of \eve s of $T$, $Tx_{n}=\lambda _{n}x_{n}$, $|\lambda _{n}|=1$, $||x_{n}||=1$, having the following properties:}
\begin{itemize}
\item[(2)] \emph{${\textrm{sp}}[x_{n}\textrm{ ; } n\geq 1]$ is dense in $X$;}

\item[(3)] \emph{for any finite subset $F$ of $\sigma  _{p}(T)\cap\T$ we have $\overline{\{x_{n} \textrm{ ; } n\geq 1\}}=\overline{\{x_{n} \textrm{ ; } n\in A_{F}\}}$, where $A_{F}=\{n\geq 0 \textrm{ ; } \lambda _{n}\not \in F\}$.}
\end{itemize}

Asumption (H') is nothing else than Assumption (H) without its first condition (1). Observe that we have proved that Assumptions (H) and (H') are again both equivalent to the fact that $T$ has perfectly spanning unimodular \eve s.

\section{Ergodicity of operators with perfectly spanning unimodular eigenvectors}
Although we now know that any \op\ on a separable Banach space with perfectly spanning unimodular eigenvectors is \fhy, we still do not know whether such an \op\ admits a \nd\ \inv\ \ga\ \mea\ with respect to which it is \erg. This question was mentioned in \cite{BG3}. Some examples seem to point out that the answer to this question should be negative, but so far no counter-example has been constructed. In this context it is interesting to note the following:

\begin{theorem}\label{th4}
If $T$ is a bounded operator on $X$ which has spanning unimodular eigenvectors, then $T$ is not \erg\ \wrt\  the \inv\ \nd\ \mea\ $m$ constructed in the proof of Theorem \ref{th1bis}. More generally, $T$ will never be \erg\ \wrt\ a  \mea\ associated to a random function
$$\Phi(\omega )=\sum_{n=1}^{+\infty }\chi_{n}(\omega )x_{n}$$ where the $x_{n}$'s are spanning eigenvectors of $T$ associated to a family of unimodular eigenvalues $\lambda _{n}$ and $(\chi_{n})_{n\geq 1}$ a sequence of independent rotation-invariant variables such that $\E(\chi_{n})=0$ and $\E(|\chi_{n}|^{2})=1$.
\end{theorem}

These \inv\  measures are in a sense the ``trivial'' ones, i.e. the ones which can be constructed without any additional assumption on the eigenvectors of $T$ (the existence of such an \inv\ \mea\ does not even imply that $T$ is \hy). When the \op\ $T$ has perfectly spanning unimodular eigenvectors with respect to a certain continuous measure $\sigma  $ on $\T$, the \mea s which are used in \cite{BG2} and \cite{BG3} to obtain ergodicity results are intrinsically different from these ones.

\begin{proof}
Let $U_{T}$ denote the isometric \op\ defined on $L^{2}(X,\mathcal{B}, m)$ by $U_{T}f=f\circ T$, $f\in L^{2}(X,\mathcal{B},m)$. If $x^{*}$ and $y^{*}$ are two elements of $X^{*}$, they belong to $L^{2}(X,\mathcal{B}, m)$. For any $n\geq 0$ we have
\begin{eqnarray*}
\langle U_{T}^{n}|x^{*}|^{2}\,,\, |y^{*}|^{2}\rangle &=&
\int_{X}|\pss{x^{*}}{T^{n}x} \, {\pss{y^{*}}{x}}|^{2}dm(x)\\
&=& \int_{\Omega }|\sum_{p\geq 0}\chi_{p}(\omega )\lambda _{p}^{n}\pss{x^{*}}{x_{p}}\,.\, \sum_{q\geq 0}{\chi_{q}(\omega )} {\pss{y^{*}}{x_{q}}}|^{2}d\P(\omega )\\
&=&  \sum_{p_{1},p_{2},q_{1},q_{2}\geq 0}
I_{p_{1},p_{2},q_{1},q_{2}}
\lambda _{p_{1}}^{n}\overline{\lambda }_{p_{2}}^{n}\pss{x^{*}}{x_{p_{1}}}
\overline{\pss{x^{*}}{x_{p_{2}}}} {\pss{y^{*}}{x_{q_{1}}}} \overline{\pss{y^{*}}{x_{q_{2}}}}
\end{eqnarray*}
where
$$I_{p_{1},p_{2},q_{1},q_{2}}=\int_{\Omega }\chi_{p_{1}}(\omega )\overline{\chi_{p_{2}}(\omega )}{\chi_{q_{1}}(\omega )}\overline{\chi_{q_{2}}(\omega )}.$$
Now $I_{p_{1},p_{2},q_{1},q_{2}}$ is non-zero if and only if $p_{1}=p_{2}$ and $q_{1}=q_{2}$ or $p_{1}=q_{2}$ and $p_{2}=q_{1}$. Thus $\langle U_{T}^{n}|x^{*}|^{2}\,,\, |y^{*}|^{2}\rangle $ is equal to
\begin{eqnarray*}
\sum_{p_{1},q_{1}\geq 0}|\pss{x^{*}}{x_{p_{1}}}|^{2}|\pss{y^{*}}{x_{q_{1}}}|^{2}+\sum_{p_{1},p_{2}\geq 0}\lambda _{p_{1}}^{n}\overline{\lambda} _{p_{2}}^{n} \pss{x^{*}}{x_{p_{1}}}
\overline{\pss{x^{*}}{x_{p_{2}}}} \overline{\pss{y^{*}}{x_{p_{1}}}} \pss{y^{*}}{x_{p_{2}}}\\
=\sum_{p\geq 0}|\pss{x^{*}}{x_{p}}|^{2}\,.\,\sum_{p\geq 0}|\pss{y^{*}}{x_{p}}|^{2}+
|\sum_{p\geq 0}\lambda _{p}^{n} \pss{x^{*}}{x_{p}}\overline{\pss{y^{*}}{x_{p}}}|^{2}.
\end{eqnarray*}
Consider now the Ces\`{a}ro sums
$$\frac{1}{N}\sum_{n=0}^{N-1}\langle U_{T}^{n}|x^{*}|^{2}\,,\, |y^{*}|^{2}\rangle=
\sum_{p\geq 0}|\pss{x^{*}}{x_{p}}|^{2}\,.\,\sum_{p\geq 0}|\pss{y^{*}}{x_{p}}|^{2}+
\frac{1}{N}\sum_{n=0}^{N-1}
|\sum_{p\geq 0}\lambda _{p}^{n} \pss{x^{*}}{x_{p}}\overline{\pss{y^{*}}{x_{p}}}|^{2}.
 $$
If $T$ were \erg\ \wrt\ $m$, this quantity would tend to 
$$\int_{X} |\pss{x^{*}}{x}|^{2}dm(x)\,.\, \int_{X} |\pss{y^{*}}{x}|^{2}dm(x)
=\sum_{p\geq 0}|\pss{x^{*}}{x_{p}}|^{2}\,.\,\sum_{p\geq 0}|\pss{y^{*}}{x_{p}}|^{2}
$$ as $N$ tends to infinity (see for instance \cite{W} for this standard characterization of ergodicity).
Hence
$$\frac{1}{N}\sum_{n=0}^{N-1}
|\sum_{p\geq 0}\lambda _{p}^{n} \pss{x^{*}}{x_{p}}\overline{\pss{y^{*}}{x_{p}}}|^{2}
 $$ would tend to zero as $N$ tends to infinity. This would imply that
$$|\sum_{p\geq 0}\lambda _{p}^{n} \pss{x^{*}}{x_{p}}\overline{\pss{y^{*}}{x_{p}}}|^{2}
$$ tends to zero as $n$ tends to infinity along a set $D$ of density $1$ (see again \cite{W}). We are going to show that it is not the case if $x^{*}$ is such that $|\pss{x^{*}}{x_{0}}|^{2}=\varepsilon >0$ and $y^{*}=x^{*}$. Since the series $\sum_{p\geq 0}|\pss{x^{*}}{x_{p}}|^{2}$ is convergent, there exists a $p_{0}$ such that for any $n\geq 0$
$$|\sum_{p>p_{0}} \lambda _{p}^{n}|\pss{x^{*}}{x_{p}}|^{2}|<\varepsilon .$$ Hence
$$|\sum_{p\geq 0} \lambda _{p}^{n}|\pss{x^{*}}{x_{p}}|^{2}|\geq |\sum_{p\leq p_{0}} \lambda _{p}^{n}|\pss{x^{*}}{x_{p}}|^{2}|-\varepsilon $$ for any $n\geq 0$. Now for any $\delta >0$
the set $D_{\delta }=\{n\geq 0\textrm{ ; for every } p\leq p_{0} \;
|\lambda _{p}^{n}-1|<\delta \}$ has positive lower density $d_{\delta }$. For any $n\in D_{\delta }$,
$$|\sum_{p\leq p_{0}}\lambda _{p}^{n} |\pss{x^{*}}{x_{p}}|^{2}| \geq \sum_{p\leq p_{0}}
|\pss{x^{*}}{x_{p}}|^{2}-\delta \sum_{p\leq p_{0}}
|\pss{x^{*}}{x_{p}}|^{2}$$ 
so that if $\delta $ is small enough,
$$|\sum_{p\geq 0}\lambda _{p}^{n} |\pss{x^{*}}{x_{p}}|^{2}|^{2}\geq \sum_{p\leq p_{0}}
|\pss{x^{*}}{x_{p}}|^{2}-2\varepsilon \geq |\pss{x^{*}}{x_{0}}|^{2}-2\varepsilon \geq \varepsilon .$$
Hence $$\frac{1}{N}\# \{n\leq N \textrm{ ; }|\sum_{p\geq 0}\lambda _{p}^{n} |\pss{x^{*}}{x_{p}}|^{2}|^{2}\geq\varepsilon \}\geq \frac{1}{2}d_{\delta }$$ for $N$ large enough, so that
$$\frac{1}{N}\# \{n\leq N \textrm{ ; }|\sum_{p\geq 0}\lambda _{p}^{n} |\pss{x^{*}}{x_{p}}|^{2}|^{2}<\varepsilon \}\leq (1-\frac{1}{2}d_{\delta }).$$ Thus 
$$|\sum_{p\geq 0}\lambda _{p}^{n} \pss{x^{*}}{x_{p}}\overline{\pss{y^{*}}{x_{p}}}|^{2}
$$ does not tend to zero along a set of density $1$. This contradiction shows that $T$ is not \erg\ \wrt\ $m$.
\end{proof}

\section{Open questions and remarks}

\subsection{Hypercyclic operators with spanning unimodular eigenvectors}
Let $T$ be a bounded \hy\ \op\ on $X$ whose \eve s associated to \eva s of modulus $1$ span a dense subspace of $X$. It is still an open question to know whether such an \op\  must be \fhy. If $T$ is a chaotic \op\ (i.e. a \hy\ \op\ which has a dense set of periodic points), then $T$ falls into this category of \ops: $T$ is chaotic \ifff\ it is \hy\ and its \eve s associated to \eva s which are $n^{th}$ roots of $1$ span a dense subspace of $X$. Thus the following question of \cite{BG1} is still unanswered: 

\begin{question}\cite{BG1}
Must a chaotic \op\ be \fhy?
\end{question}
\par\smallskip
It is an intriguing fact that all \ops\ which are known to be \hy\ and to have spanning unimodular \eve s have in fact perfectly spanning unimodular \eve s. Hence a natural way to prove (or disprove) the conjecture that all \hy\ \ops\ with
spanning unimodular \eve s are \fhy\ would be to answer the following question:

\begin{question}\label{q2}
If  $T$ is a \hy\ \op\ on $X$ whose \eve s associated to \eva s of modulus $1$ span a dense subspace of $X$, is is true that the unimodular \eve s of $T$ are perfectly spanning?
\end{question}

A related question of \cite{Fl} is interesting in this context:

\begin{question}\cite{Fl}\label{q22}
Does there exist a bounded \hy\ \op\ $T$ on $X$ whose unimodular point spectrum consists of a countable set $\{\lambda _{n} \textrm{ ; } n\geq 1\}$, and which is such that the \eve s associated to the \eva s $\lambda _{n}$ span a dense subspace of $X$?
\end{question}

\subsection{Existence of frequently hypercyclic and chaotic operators on complex Banach spaces with an unconditional Schauder decomposition}

Let $X$ be a complex separable infinite-dimensional Banach space $X$ with an unconditional Schauder decomposition.
This means that there exists a sequence $(X_{n})_{n\geq 0}$ of closed subspaces of $X$ such that any $x\in X$ can be written in a unique way as an unconditionnally convergent series $x=\sum_{n\geq 0}x_{n}$, where $x_{n}$ belongs to $X_{n}$ for any $n\geq 0$, and there is no loss of generality in supposing that all the subspaces $X_{n}$ are infinite-dimensional. The main result of \cite{DFGP} states that there exists a bounded \op\ on $X$ which is \fhy\ and chaotic.
This result was motivated by the fact that  any infinite-dimensional Banach space supports a \hy\ \op\ (\cite{A}, \cite{B}), but that the corresponding statement for \fhy\ \ops\ is not true \cite{S}: if $X$ is a \sep\ complex hereditarily indecomposable space (like the space of Gowers and Maurey \cite{GM}), then there is no \fhy\ \op\ on $X$. Recall that a Banach space $X$ is said to be hereditarily indecomposable if no pair of 
closed infinite-dimensional subspaces
$Y$ and $Z$ of $X$ form a topological direct sum $Y\oplus Z$. Also \cite{BMP} there are no chaotic \ops\ on a complex hereditarily indecomposable Banach space.
The \ops\ constructed in \cite{DFGP} are perturbations of a diagonal \op\ with unimodular coefficients by a vector-valued nuclear backward shift. In \cite{DFGP} we first construct such \ops\ on a Hilbert space, prove that they have perfectly spanning unimodular \eve s, and then transfer them to our Banach space $X$.
 This result can also be obtained as a consequence of Theorem \ref{th1bis}: the \eve s can be directly computed, and if at each step of the constuction we take the perturbation of the diagonal coefficients to be small enough, the \op\ satisfies assumption (H). 
The proof of \cite{DFGP} is, however, much simpler.

\end{document}